\documentclass[12pt,reqno]{amsart}
\usepackage{amssymb,amsthm,amscd}
\usepackage[utf8]{inputenc}

\newtheorem{theorem}{Theorem}
\newtheorem{proposition}{Proposition}
\newtheorem{lemma}{Lemma}
\newtheorem{corollary}{Corollary}

\theoremstyle{definition}

\newtheorem{definition}{Definition}
\newtheorem{example}{Example}
\newtheorem{remark}{Remark}

\def\Var{\mathrm{Var}}

\def\di{\mathrm{di}}
\def\Com{\mathrm{Com}}
\def\Lie{\mathrm{Lie}}
\def\Leib{\mathrm{Leib}}
\def\Cur{\mathrm{Cur}}
\def\LSym{\mathrm{LSym}}
\def\Perm{\mathrm{Perm}}
\def\SLS{\mathrm{SLS}}
\def\Nov{\mathrm{Nov}}
\def\Der{\mathrm{Der}}
\def\Ker{\mathop{\fam 0 Ker}\nolimits}
\def\wt{\mathop{\fam 0 wt}\nolimits}

\let\<\langle
\let\>\rangle

\numberwithin{equation}{section}

\title[On the embedding into differential algebras]{On the embedding of left-symmetric algebras into differential Perm-algebras}

\author{P. S. Kolesnikov}

\author{B. K. Sartayev}

\thanks{This work was supported by the Program of fundamental scientific researches of the Siberian Branch of Russian Academy of Sciences, I.1.1, project 0314-2019-0001}

\address{Sobolev Institute of Mathematics, Novosibirsk, Russia}

\begin{document}

\begin{abstract}
Given an associative algebra satisfying the left commutativity 
identity $abc=bac$ (Perm-algebra) with a 
derivation $d$, the new operation $a\circ b = a d(b)$
is left-symmetric (pre-Lie). We derive necessary and 
sufficient conditions for a left-symmetric algebra 
to be embeddable into a differential Perm-algebra.
\end{abstract}

\maketitle

\section{Introduction}

The class of left-symmetric algebras (also known as pre-Lie algebras) 
initially appeared in deformation theory and geometry 
(\cite{Gerst}, \cite{Kosz}, \cite{Vin}). 
By definition, a left-symmetric algebra is a linear space with one 
bilinear multiplication 
satisfying the identity 
\[
(a,b,c) = (b,a,c),
\]
where $(a,b,c)= (ab)c-a(bc)$ is the associator. 

If, in addition, the right commutativity holds, i.e., 
\[
(ab)c=(ac)b,
\]
then such a system is known as a Novikov algebra. 
Novikov algebras appeared
in formal variational calculus \cite{GD79}
and, independently, as a tool for studying Poisson 
brackets of hydrodinamic type \cite{BN83}. 
The structure theory of Novikov 
algebras is well-developed, see, e.g., 
\cite{Liu2019} and references therein.

For example, if $A$ is an associative and commutative 
algebra with a derivation $d$ then the same space $A$
relative to the new product 
\begin{equation}\label{eq:DerProd}
a\circ b = a d(b), \quad a,b\in A,
\end{equation}
is a Novikov algebra.
As shown in \cite{BCZ2018}, this example is generic: every 
Novikov algebra embeds into a differential commutative algebra. 
An alternative way to prove the latter is to apply 
the following observation: the free Novikov algebra 
$\Nov \<X\>$
generated by a set $X$ embeds into the algebra 
$\Com\Der\<X,d\>$
of differential polynomials in $X$ \cite{DzhLofwall2002}, 
and it is easy to show \cite{SK2021_arxiv} that every 
homomorphic image of $\Nov \<X\>$ also embeds 
into a commutative differential algebra. 

In this paper, we study the class  of left-symmetric algebras obtained from non-commutative 
differential algebras by means of the operation 
\eqref{eq:DerProd}. 
Namely, if $A$ is an associative algebra satisfying 
the identity of left commutativity 
$abc = bac$ (i.e., $A$ is a Perm-algebra) 
then for every derivation $d$ on $A$
the operation \eqref{eq:DerProd} turns $A$ into 
a left-symmetric algebra $A^{(d)}$. 

Let us call a left-symmetric algebra {\em special}
if it can be embedded into a differential Perm-algebra
via \eqref{eq:DerProd}.
Although Perm-algebras are very close to commutative ones, 
the description of the class of special left-symmetric algebras 
differs from that of Novikov algebras. 
The most important difference is that it is not a variety: a homomorphic image of a special left-symmetric algebra 
may not be special. We show the class 
of special left-symmetric algebras to be a quasi-variety.
It is easy to see that every special left-symmetric algebra satisfies 
$((ax)y)b=((ay)x)b$ and $(x,ya,b)=(y,xa, b)$.
The corresponding variety of {\em SLS-algebras}
contains the class of special left-symmetric algebras
and we find explicitly the set of quasi-identities
defining special algebras within the SLS-algebras.
Finally, we prove that the class of SLS-algebras 
is the smallest variety containing special left-symmetric algebras. The latter follows from the observation 
that the free SLS-algebra is special.

We believe that the same approach as developed in this paper 
will be useful for finding the complete list of special identities 
for Gelfand--Dorfman algebras. In that case (see \cite{SK2021_arxiv}) the class of special algebras does form 
a variety, and it is known that there are two idependent 
identities of degree~4.

Throughout the paper, $\Bbbk $ is an arbitrary base field.
We will use the following notations. 
If $\Var$ is a variety of algebras then 
$\Var\<X\>$ stands for the free algebra in $\Var $ generated by a set~$X$.
We will use the same symbol $\Var $ to denote 
the operad goverining the variety $\Var $. 

For example, if $\Com $ denotes the variety of 
associative and commutative algebras then 
$\Com \<X\> = \Bbbk [X]$ is the ordinary polynomial algebra. By $\LSym $ we denote the variety of 
left-symmetric algebras, $\Nov \subset \LSym $ stands 
for Novikov algebras, etc.

The class of pairs $(A,d)$, $A\in \Var$, $d$ is a derivation on $A$ is also a variety denoted 
$\Var\Der$. The free algebra in $\Var\Der$ generated by a set $X$ is denoted $\Var\Der \<X,d\>$. Here we include $d$ into the notation for convenience. 
As a $\Var$-algebra, $\Var\Der\<X,d\>$ is isomorphic 
to $\Var \<X^{(\omega )}\>$, where 
\[
X^{(\omega )} = X\cup X' \cup X''\cup \dots \cup X^{(n)}\cup \dots, 
\]
$X^{(n)} = \{x^{(n)} \mid x\in X \}$, $n\ge 0$, 
are  disjoint copies of~$X$. 
The derivation $d$ acts on the 
generators of 
$\Var \<X^{(\omega )}\>$ as $d(x^{(n)})=x^{(n+1)}$.

The interest to the study of differential $\Perm$-algebras
has one more motivation. 
The operad $\Perm $  plays a specific role in the theory 
of dialgebras \cite{Loday2001} and their operads (called replicated operads) \cite{GubKol2013}, \cite{GuoBai2020}: 
given an operad $\Var$, the Manin white product 
of operads $\Perm \circ \Var $ is exactly the operad 
governing the class $\di\Var $ of $\Var$-dialgebras.
Note that for $\Perm $ the Manin white product 
coincides with the Hadamard product of operads. 

On the other hand, the operad $\Nov $ has its own 
distinguished role in the combinatorics of derivations
on non-associative algebras \cite{KSO2019}: given a binary operad $\Var$, 
the Manin white product 
$\Nov\circ \Var$ is the operad governing 
the class of derived $\Var $-algebras. 
The latter are obtained from $\Var $-algebras with 
 a derivation $d$ relative to the new operations 
\[
a\prec b = a d(b),\quad a\succ b = d(a)b
\]
(for each of binary products in $\Var $, if there are more than one operation).

In the problem we consider in this paper, $\Perm $ and 
$\Nov $ meet each other since we have to study differential $\Perm$-algebras. This is why the theory of dialgebras (especially, Novikov dialgebras) 
provides us with effective tools for studying speciality 
of left-symmetric algebras.

The paper is organized as follows. 
 
In Section~2 we recall the definition of a $\Var$-dialgebra and prove an analogue of the embedding theorem from \cite{BCZ2018} for Novikov dialgebras. 
This is done in a routine way by means of the general methods described in \cite{GubKol2013} (see also \cite{GubKol2014}).  Since every 
Novikov dialgebra is in particular a left-symmetric algebra, we obtain that an embedding of a left-symmetric 
algebra into a differential $\Perm$-algebra is equivalent 
to the embedding into a Novikov dialgebra.

Section~3 is devoted to the study of the forgetful 
functor from $\di\LSym $ to $\LSym $. 
It turns out that the left adjoint functor has a very 
natural description: given an algebra $A\in \LSym$, we may define the structure of its universal enveloping left-symmeric dialgebra on the tensor algebra $T(A)$. 

In Section~4, we prove that the free SLS-algebra is special. Namely, the subalgebra of $\Perm\Der\<X,d\>$
generated by the set $X$ relative to the operation 
$a\circ b = a d(b)$ is isomorphic to the free left-symmetric algebra satisfying two identities of SLS-algebras mentioned above. In contrast to 
\cite{DzhLofwall2002}, we do not find the monomial 
basis of $\SLS\<X\>$, but still we can describe its 
linear basis explicitly as a subset of $\Perm\Der\<X,d\>$.

Finally, in Section~5 we derive the necessary conditions 
of speciality for an SLS-algebra in the form of quasi-identities. Then we prove that these conditions are 
also sufficient. Given an $\SLS$-algebra $A$,
we construct an appropriate 
quotient of its universal enveloping $\di\LSym$-algebra 
 constructed in Section~3 to get a Novikov dialgebra envelope of~$A$. The results of Section~2 then lead us to the desired conclusion.

It is easy to construct examples of non-special SLS-algebras, 
however, the free SLS-algebra is special (Section~4). 
Hence, the class of special left-symmetric 
algebras does not form a variety.

\section{Novikov dialgebras and differential Perm-algebras}

Let $\Perm $ denote the class of associative algebras 
satisfying the identity 
\[
xyz-yxz = 0.
\]
The operad governing the variety of such algebras is also 
denoted $\Perm $. It is clear (see \cite{Chapoton01}) that 
$\dim \Perm (n)=n$, and the composition rule is easy to 
describe \cite{Kol2008smz}.

Given a binary quadratic operad $\mathcal O$, the Hadamard product of operads $\mathcal O\otimes \Perm $ 
coincides with the Manin 
white product \cite{Vallette}, and the class of algebras 
governed by $\mathcal O\otimes \Perm $ is known as 
the variety of $\di\,\mathcal O$-algebras (replicated 
$\mathcal O$-algebras \cite{GuoBai2020}, or $\mathcal O$-dialgebras \cite{Kol2008smz}). 
In particular, if $\Com$ and $\Lie $ denote the operads governing commutative and Lie algebras then 
$\di\Com$ and $\di\Lie$ are isomorphic to $\Perm $ 
and $\Leib$, respectively, where $\Leib $ is the operad 
of Leibniz algebras.

If $\dim O(2)=2$ (one binary product with no symmetry) 
then the operad $\di\,\mathcal O$ 
is generated by two operations $\vdash $ and $\dashv $. 
The defining relations of $\di\,\mathcal O$ are easy to derive 
from those of $\mathcal O$
\cite{Kol2008smz}. Namely, every $\di\,\mathcal O$-algebra
satisfies so-called {\em 0-identities}
\begin{equation}\label{eq:zero-id}
\begin{gathered}
(x_1\dashv x_2)\vdash x_3 = (x_1\vdash x_2)\vdash x_3,
\\
x_1\dashv (x_2\vdash x_3) = x_1\dashv (x_2\dashv x_3) .
\end{gathered}
\end{equation}
Moreover, for every multi-linear defining identity 
$f(x_1,\dots , x_n) = 0$ of $\mathcal O$
and for every $i=1,\dots, n$
one should claim 
\[
f_i(x_1,\dots , x_n) = f(x_1,\dots,\dot x_i ,\dots,  x_n) = 0
\]
to hold on $\di\,\mathcal O$. 
Here $f(x_1,\dots,\dot x_i ,\dots,  x_n)$
is derived from $f$ in the following way: in each nonassociative monomial 
$(x_{j_1}\dots x_{j_n})$ the initial binary 
 product is replaced with the products 
$\vdash$ and $\dashv$ so that the horizontal 
dashes point to the fixed variable $x_i$.

\begin{example}\label{exmp:diNov}
Let $\Nov$ stand for the variety of Novikov algebras as well as for the corresponding operad. Then 
$\di\Nov$ is defined by \eqref{eq:zero-id} together with 
\begin{equation}\label{eq:diLSym}
 \begin{gathered}
(x_1\dashv x_2)\dashv x_3 - x_1\dashv (x_2\dashv x_3)
= (x_2\vdash x_1)\dashv x_3 - x_2\vdash (x_1\dashv x_3),\\
(x_1\vdash x_2)\vdash x_3 - x_1\vdash (x_2\vdash x_3)
= (x_2\vdash x_1)\vdash x_3 - x_2\vdash (x_1\vdash x_3),
 \end{gathered}   
\end{equation}
\begin{equation}\label{eq:diRCom}
 \begin{gathered}
(x_1\dashv x_2)\dashv x_3 = (x_1\dashv x_3)\dashv x_2, \\ 
(x_1\vdash x_2)\dashv x_3 = (x_1\vdash x_3)\vdash x_2.
 \end{gathered}   
\end{equation}
A linear space $N$ equipped with bilinear operations 
$\vdash $ and $\dashv $ is a Novikov dialgebra if and only if 
\eqref{eq:zero-id}--\eqref{eq:diRCom} hold for all $x_1,x_2,x_3\in N$.
\end{example}

Obviously, every $\Perm$-algebra ($\di\Com$-algebra) 
with a derivation $d$ turns into a Novikov dialgebra ($\di\Nov$-algebra) relative to the new operations 
$x\vdash y = xd(y)$, 
$x\dashv y = d(y)x$.
The embedding statement may be derived by means of 
the general construction from \cite{GubKol2013} 
that allows us to reduce the problem on a Novikov 
dialgebra to an ``ordinary'' Novikov algebra
by means of conformal algebras (pseudo-algebras). 
One may refer 
to \cite{BDK2001} for details on pseudo-algebras, 
but for our puproses the only essential instance 
is that every pseudo-algebra may be turned into 
 a dialgebra \cite{Kol2008smz}.

\begin{theorem}\label{thm:diNov-Perm}
 Every Novikov dialgebra $(N,\vdash,\dashv )$ may be embedded into a differential Perm-algebra 
 in such a way that 
 \[
  x\vdash y = xd(y),\quad x\dashv y = d(y)x
 \]
for $x,y \in N$.
\end{theorem}

\begin{proof}
If $N $ is a Novikov dialgebra
with operations $\vdash $ and $\dashv $
then 
\[
N_0 = \mathrm{span}\,\{a\vdash b - a\dashv b \mid a,b\in N\}
\]
is an ideal in $N$ and $\bar N = N/N_0$ is a Novikov algebra. 
The space $N$ is a Novikov bimodule 
over $\bar  N $
relative to the action
\[
\bar a \circ b = a\vdash b,\quad a\circ \bar b = a\dashv b,
\]
for $a,b\in N$.
The corresponding split null 
extension $\hat N= \bar N\ltimes N$ is a Novikov algebra \cite{Pozh2009} in which $N\circ N = 0$. 

Let $H$ be a cocommutative bialgebra with a primitive element 
$T$ (e.g., $H=\Bbbk [T]$). Consider 
the current pseudo-algebra 
$\Cur\hat N = H\otimes \hat N$. 
Then $\Cur\hat N$ is a Novikov dialgebra relative to the operations  
\[
(f\otimes x)\vdash (g\otimes y) = \varepsilon(f)g\otimes x\circ y,
\quad 
(f\otimes x)\dashv (g\otimes y) = \varepsilon(g)f\otimes x\circ y
\]
for $x,y\in \hat N$, $f,g\in H$, where $\varepsilon $ is the counit in $H$.

The dialgebra $N$ embeds into $\Cur\hat N$ via
\[
a\mapsto \hat a = 1\otimes \bar a + T\otimes a,\quad a\in N.
\]

Let $(B,\cdot, d)$ be a commutative differential algebra
which contains $\hat N$ in such a way that 
\[
x\cdot d(y) = x\circ y, \quad x,y\in \hat N.
\]
Then $C = \Cur B$ may be considered as a commutative dialgebra (i.e., $\Perm $-algebra) and it is easy to check that 
$\hat d = \mathrm{id}\otimes d$ is a derivation of $C$.
Thus we have 
\[
N \subset \Cur \hat N \subset \Cur B = C,
\]
and it is easy to check that the operations
$\vdash$, $\dashv$ on $N$
are related with the operations $\cdot$, $\hat d$ on $C$ 
in the desired way. Indeed, if $a,b\in N$,
$\hat a,\hat b\in \Cur \hat N $ then the product 
of $\hat d(\hat b)$ and $\hat a$ in the $\Perm$-algebra 
$C$ may be calculated as follows:
\begin{multline*}
\hat d(\hat b)\hat a
= (1\otimes d(\bar b) + T\otimes d(b))(1\otimes \bar a + T\otimes a)
= 1\otimes d(\bar b)\cdot\bar a + T\otimes d(\bar b)\cdot a)
\\
= 1\otimes (\bar a\circ \bar b) + T\otimes (a\circ \bar b)
= 1\otimes \overline{a\dashv b} + T\otimes (a\dashv b)
=\widehat {a\dashv b}.
\end{multline*}
In a similar way, 
$\widehat {a\vdash b} = \hat a\hat d(\hat b)$
in $C$ for all $a,b\in N$.
\end{proof}

\begin{corollary}\label{cor:Dialg-embedding}
A left-symmetric algebra $(A,\circ)$ embeds into a 
differential Perm algebra in such a way 
that $a\circ b = ad(b)$ for $a,b\in A$
if and only if $(A,\circ )$ embeds into a Novikov dialgebra
in such a way that 
$a\vdash b = a\circ b$
for $a,b\in A$.
\end{corollary}

\section{Universal enveloping left-symmetric dialgebras}

According to the general rule of constructing defining identities of dialgebras, 
every left-symmetric dialgebra $(A,\vdash, \dashv)$ is a left-symmetric algebra 
relative to the operation $\vdash $. 
Thus we have a forgetful functor 
$ \di\LSym \to \LSym $. In this subsection, we explicitly construct its left adjoint functor 
$\mathcal D: \LSym \to \di\LSym $ which will be applied in the sequel. 

Let $(A,\circ )$ be a left-symmetric algebra. 
Consider the (associative) tensor algebra $T(A)$ of the space $A$ (without the identity):
\[
T(A) = A\oplus A^{\otimes 2} \oplus A^{\otimes 3} 
\oplus \dots .
\]
We will write $a_1\dots a_n$ for the element 
$a_1\otimes \dots \otimes a_n\in A^{\otimes n}$ 
for brevity. 

Define the map $M: T(A)\to A$ as follows: 
\[
 M(a_1\dots a_n) = ((\dots  ((a_1\circ a_2)\circ a_3 ) \circ \dots) \circ a_n), \quad a_i\in A.
\]
Let us define a binary operation 
$\vdash $ on the space $T(A)$ by 
induction on the length of the second factor:
\begin{equation}\label{eq:tensor-vdash}
\begin{gathered}
u\vdash a = M(u)\circ a,\quad u\in T(A),\ a\in A,\\
u\vdash (wa) = (u\vdash w)a + w (M(u)\circ a) - wM(u)a,
\end{gathered}
\end{equation}
for $u,w\in T(A)$, $a\in A$.
Finally, define a binary operation $\dashv $ on $T(A)$ as 
\begin{equation}\label{eq:tensor-dashv}
 u\dashv w = uM(w),\quad u,w\in T(A).
\end{equation}

\begin{lemma}\label{lem:Mprod}
For every $u,w\in T(A)$ we have 
\[
u\vdash w = M(u)\vdash w,
\quad 
u\dashv w = u\vdash M(w),
\]
\[
M(u\vdash w) = M(u\dashv w) = M(u)\circ M(w).
\]
\end{lemma}

\begin{proof}
The desired equalities are obvious for the operation $\dashv $, as well as for $w=a\in A$. Proceed by induction
on the length of $w$:
\begin{multline}\nonumber
M(u\vdash wa) = M(u\vdash w)\circ a + M(w)\circ (M(u)\circ a)
-(M(w)\circ M(u))\circ a \\
=(M(u)\circ M(w))\circ a - (M(w),M(u),a) 
= M(u)\circ (M(w)\circ a) \\
= M(u)\circ M(wa)
\end{multline}
by the left symmetry.
\end{proof}

\begin{proposition}\label{prop:tensor-diLS}
For every left-symmetric algebra $(A,\circ)$, the space 
$T(A)$ equip\-ped with the operations \eqref{eq:tensor-vdash}, \eqref{eq:tensor-dashv}
is a left-symmetric dialgebra. 
\end{proposition}

\begin{proof}
It is enough to check that 
\eqref{eq:zero-id} and \eqref{eq:diLSym} hold 
on $(T(A),\vdash,\dashv)$. 

Let us start with \eqref{eq:zero-id}.
Consider $u,v,w\in T(A)$ and note that
\begin{multline}\nonumber
 (u\vdash v)\vdash w  = 
 M(u\vdash v)\vdash w \\
 = (M(u)\circ M(v))\vdash w =  M(u\dashv v)\vdash w
 = (u\dashv v)\vdash w 
\end{multline}
by Lemma \ref{lem:Mprod}.
The remaining 0-identity holds by similar reasons.
 
Next, proceed to \eqref{eq:diLSym}. Let us state the calculation 
for the more complicated case: show that the ``right'' associator
\[
(u\vdash v)\vdash w - u\vdash (v\vdash w) 
\]
is symmetric relative to the exchange of $u$ and $v$.
Indeed, for $w=c\in A$ it is enough to apply Lemma~\ref{lem:Mprod}
and the left symmetry of $A$. For the general case, apply induction 
on the length of $w$. Since we may replace $u$ and $v$ with 
$M(u)$ and $M(v)$ by Lemma~\ref{lem:Mprod},
consider $u=a$, $v=b$ ($a,b\in A$):
\begin{multline}\nonumber
(a\vdash b)\vdash wc - a\vdash (b\vdash wc) \\
= ((a\circ b)\vdash w )c + w((a\circ b)\circ c) - w(a\circ b)c \\
- a\vdash ( (b\vdash w)c+w(b\circ c)-wbc) \\
=  ((a\circ b)\vdash w )c + w((a\circ b)\circ c) - w(a\circ b)c \\
-(a\vdash (b\vdash w))c - (b\vdash w)(a\circ c) + (b\vdash w)ac \\
-(a\vdash w)(b\circ c) -w(a\circ (b\circ c)) + wa(b\circ c) \\
+ (a\vdash wb)c + wb(a\circ c) - wbac .
\end{multline}
Let us remove those summands that already form symmetric expressions (relative 
to the exchange of $a$ and $b$) and expand the last $(a\vdash wb)$: 
\begin{multline}\nonumber
 - w(a\circ b)c - (b\vdash w)(a\circ c) + (b\vdash w)ac -(a\vdash w)(b\circ c) + wa(b\circ c) \\
  + (a\vdash w)bc + w(a\circ b)c - wabc + wb(a\circ c)
- wbac \\
=
(a\vdash w)bc  + (b\vdash w)ac 
-(a\vdash w)(b\circ c) - (b\vdash w)(a\circ c) \\
+ wa(b\circ c) + wb(a\circ c)
- wbac  - wabc .
\end{multline}
The expression obtained has the desired symmetry.

The remaining relation from \eqref{eq:diLSym} can be proved in a similar but simpler way.
\end{proof}

Denote the left-symmetric dialgebra $(T(A),\vdash ,\dashv)$ by $\mathcal D(A)$. 
The injection $\iota : A\to \mathcal D(A)$, $\iota(a)=a$, is a homomorphism  from $A$ to 
$\mathcal D(A)^\circ =(\mathcal D(A), \vdash ) $
by the definition. 

\begin{proposition}\label{prop:NormalU-LS}
Let $(A,\circ )$ be a left-symmetric algebra.
Then for every left-symmetric dialgebra $B$ and for every homomorphism $\tau : A\to B^\circ $ 
there exists unique homomorphism of dialgebras $\varphi : \mathcal D(A)\to B$
such that $\varphi \iota = \tau $. 
\end{proposition}

\begin{proof}
It is enough to check that the map 
\begin{equation}\label{eq:diLS-normal}
\varphi(a_1a_2\dots a_{n-1}a_n) = 
(\dots  ((x_1\dashv x_2) \dashv \dots \dashv x_{n-1}) \dashv  x_n),\quad x_i=\tau(a_i),
\end{equation}
is a homomorphism of dialgebras preserving $\vdash $ and $\dashv $.
The latter follows from the following observation:
for all $w\in T(A)$, $x\in B$ we have
\[
\varphi (w)\vdash x = \tau(M(w))\vdash x,
\quad 
x\dashv \varphi (w) = x\dashv \tau(M(w))
\]
by \eqref{eq:zero-id}.
It remains to prove 
\[
\varphi (a\vdash w)  =\tau(a)\vdash \varphi (w),\quad a\in A,
\]
by induction on the length of $w\in T(A)$, 
and apply Lemma~\ref{lem:Mprod}.
\end{proof}

Thus $\mathcal D(A) $ is the universal enveloping left-symmetric dialgebra
of of a left-symmetric algebra $A$. 
As we can see, every $A$ embeds into its universal dialgebra envelope 
which is just $T(A)$ as a linear space.
The functor $\mathcal D(\cdot )$ is left adjoint to the forgetful functor $\di\LSym \to \LSym$ which is induced by the morphism 
of operads $x_1\circ x_2 \mapsto x_1\vdash x_2$.

Note that the same sort of functor acts on $\di\Nov$: every 
Novikov dialgebra is in particular a left-symmetric algebra 
relative to the operation $\vdash $. So one may construct 
a universal enveloping Novikov dialgebra for a left-symmetric 
algebra $A$, but the problem is to distinguish 
those algebras that are embedded into such envelopes.
We will completely resolve this problem in Section~\ref{sec:Crit}.

\section{A linear basis of the free SLS-algebra}\label{sec:BasisFree}

Let $P$ be a Perm-algebra with a derivation $d$. 
As above, denote $d(x)$ by $x'$ for $x\in P$. 
A new binary operation 
\begin{equation}\label{eq:DiffProd}
x\circ y = xy', \quad x,y\in P,    
\end{equation}
turns $P$ into a left-symmetric algebra. Moreover, 
the following identities hold on $P$:
\begin{gather}
    ((x\circ y)\circ z)\circ u = ((x\circ z)\circ y)\circ u ,\label{eq:SLS-1} \\
    (x,y\circ z, u) = (y,x\circ z, u), \label{eq:SLS-2}, 
\end{gather}
where $(a,b,c)$ stands for the associator 
$(a\circ b)\circ c - a\circ (b\circ c)$.

\begin{definition}\label{defn:SLS-algebra}
A left-symmetric algebra $A$ satisfying 
\eqref{eq:SLS-1} and \eqref{eq:SLS-2} is said to be an {\em SLS-algebra}.
\end{definition}

Let $X$ be a nonempty set.
In this section, we prove that the free SLS-algebra $\SLS\langle X\rangle $ embeds into the free differential Perm-algebra 
$\Perm\Der \langle X,d\rangle $ relative to the operation 
\eqref{eq:DiffProd}. 
Namely, define a homomorphism 
\[
\tau :\SLS\langle X\rangle\to \Perm\Der\langle X,d\rangle 
\]
such that 
\[
\tau (x)=x,\ x\in X, \quad \tau (f\circ g) = \tau(f)\tau(g)'.
\]
The main purpose of this section is to prove that $\tau $ is injective.

A Novikov algebra is in particular an SLS-algebra. Hence, there exists a homomorphism 
\[
\phi : \SLS \langle X \rangle  \to \Nov \langle X\rangle . 
\]
Choose a monomial basis $\bar N$ in $\Nov \langle X\rangle $
and for every $u\in \bar N$ choose its monomial pre-image in $\SLS\langle X\rangle $.
Denote the set of such pre-images by $N$. There exists a monomial basis $B$
of $\SLS \langle X \rangle $ such that $N\subset B$. Let us choose and fix 
such a set~$B$.

Every non-associative word $u$ in $X$ can be written as 
\begin{equation}\label{eq:RNform}
u = L(u_1,\dots, u_k, x):=u_{1}\circ (u_{2} \circ ( \dots \circ (u_{k}\circ x)\dots )),
\end{equation}
where $u_{j}$ are non-associative words, $x\in X$.
Let $k$ be the {\em length\/} of $u$.
For a linear combination of \eqref{eq:RNform}, the length 
is the maximal length of its summands.

\begin{lemma}\label{lem:Lemma2}
Every element of $\SLS\langle X\rangle $
may be written as a linear combination 
of words of the form \eqref{eq:RNform} with $u_j\in N$, $x\in X$.
\end{lemma}

\begin{proof}
It follows from \eqref{eq:SLS-1} that if 
$\phi (a)=\phi(b)$ then 
$a\circ u = b\circ u$ for all 
$a,b,u\in \SLS\langle X\rangle $.
This observation makes the lemma obvious.
\end{proof}

There exist various presentations of an element $f\in \SLS\langle X\rangle $ 
by linear combinations of \eqref{eq:RNform} as in Lemma~\ref{lem:Lemma2}.
However, for every nonzero $f$ we may choose a presentation 
with minimal length. In this way, we obtain a well-defined 
length function 
\[
\ell :\SLS\langle X\rangle\setminus \{0\} \to \mathbb Z_+
\]
which depends only on the choice of~$B$.

Recall the notion of {\em weight\/} in a 
free commutative differential algebra 
\cite{DzhLofwall2002}. 
Let $X^{(\omega )}$ stand for the disjoint union 
$X\cup X' \cup X'' \cup\dots $, then $\Com\Der\langle X,d\rangle $
as a commutative algebra coincides with 
the polynomial algebra $\Com\langle X^{(\omega )}\rangle $. 
Define 
\[
\wt(x^{(j)}) =j-1,\quad x\in X,\ j\ge 0, 
\]
and set $\wt(uv) = \wt(u)+\wt(v)$ 
for all monomials $u,v\in \Com\langle X^{(\omega )}\rangle $.

As shown in \cite{DzhLofwall2002}, the subspace $W$ spanned in 
$\Com\langle X^{(\omega )}\rangle$ 
by all monomials of weight~$-1$ is isomorphic to 
the free Novikov algebra $\Nov\langle X\rangle $
relative to the map 
\[
\Nov\langle X\rangle \to \Com \langle X^{(\omega )}\rangle 
\]
defined in the same way as~$\tau $.
We will not distinguish notations for the maps 
presenting horizontal arrows in the diagram below.
\[
\begin{CD}
\SLS\langle X\rangle  @>\tau >> \Perm\langle X^{(\omega)}\rangle \\
@V\phi VV @VVV \\
\Nov\langle X\rangle @>\tau >> \Com\langle X^{(\omega)}\rangle
\end{CD}
\]

Hence, for every monomial $u\in W$
we may find its unique pre-image 
$\tau^{-1}(u) \in \Nov \langle X\rangle $, 
write $\tau^{-1}(u)$ as a linear combination of $N$, and consider 
the same combination as an element of $\SLS\langle X\rangle $ since 
$N\subset B$. Thus we obtain a well-defined map
$\tau^{-1}:W\to \Bbbk N \subset \SLS\langle X\rangle $ which depends only 
on the choice of~$B$.

If $u\in \SLS\langle X\rangle $ 
is a monomial of the form \eqref{eq:RNform}
then 
\[
 \tau(u) = 
 \tau(u_{1})\tau (u_{2}) \dots \tau(u_{k}) x^{(k)} + g_u
 \in \Perm \langle X^{(\omega )} \rangle 
\]
where all monomials in $g_u$ end with $x^{(j)}$ for $j<k$.

Recall that a linear basis of 
$\Perm\langle X^{(\omega )}\rangle $ 
consists of all monomials of the form $ux^{(j)}$,
where $j\ge 0$, $x\in X$, and $u$ is from a linear basis 
of the free associative commutative algebra $\Com\langle X^{(\omega )}\rangle $ (see, e.g., \cite{Chapoton01}). 
In particular, associative words with different last letters are linearly 
independent in the free Perm-algebra.

Assume $\Ker\tau \ne 0$ in $\SLS\langle X\rangle $. 
Choose a nonzero element $f\in \Ker\tau $
of minimal length, $\ell (f) = k$. 
Without loss of generality we may suppose that all monomials 
in $f$ in a presentation given by Lemma~\ref{lem:Lemma2}
end with the same letter $x\in X$:
\[
f = \sum\limits_{j}\alpha _j 
L(u_{1j}, u_{2j}, \dots ,u_{kj}, x)
+ \tilde f,
\]
where $\ell (\tilde f)<k$ or $\tilde f=0$. 
Then 
\[
\tau(f) = \sum\limits_{j}\alpha _j\tau(u_{1j}) \dots \tau(u_{kj}) x^{(k)}
+ h,
\]
where all monomials in $h$ end with $x^{(j)}$ for $j<k$.
Hence, 
\[
\sum\limits_{j}\alpha _j\tau(u_{1j}) \dots \tau(u_{kj}) = 0
\]
in $\Com\Der\langle X,d\rangle = \Com\langle X^{(\omega )}\rangle $.

The following statement leads us to a contradiction 
with the choice of~$f$.

\begin{proposition}\label{prop:Prop3}
Let 
\[
F = \sum\limits_{j}\alpha _j w_{1j}\otimes  \dots\otimes  w_{kj} \in W^{\otimes k}, \quad \alpha_j\in \Bbbk ,
\]
turn into zero by multiplication, i.e., 
\[
\mu(F) := \sum\limits_{j}\alpha _j w_{1j} \dots w_{kj} = 0\in \Com\langle X^{(\omega )} \rangle .
\]
Then $\ell (g)<k$ for 
\[
g = \sum\limits_{j}\alpha _j 
L(u_{1j}, u_{2j}, \dots , u_{kj}, x),
\]
where $u_{ij}=\tau^{-1}(w_{ij})$, $x\in X$.
\end{proposition}

Before we proceed to the proof, let us state two examples to 
explain the meaning of Proposition~\ref{prop:Prop3}.

\begin{enumerate}
\item 
$F = w_1\otimes w_2-w_2\otimes w_1\in W^{\otimes 2}$ 
turns into zero by multiplication. 
Then for $u_i= \tau^{-1} (w_i)\in \Bbbk N\subset \SLS\langle X\rangle  $, $i=1,2$, 
we have 
\[
u_1\circ (u_2\circ x) - u_2\circ (u_1\circ x) 
= (u_1\circ u_2)\circ x - (u_2\circ u_1)\circ x
\]
by left symmetry. The right-hand side is of length $\le 1$ 
which is smaller than $k=2$.

\item 
$F = w_1\otimes w_2w_3' - w_2\otimes w_1w_3' \in W^{\otimes 2}$ turns 
into zero by multiplication (here $w_i \in W$). 
Then for 
$u_i=\tau^{-1}(w_i)$,
we have 
$u_i\circ u_3\equiv \tau^{-1}(w_iw'_3)\pmod {\Ker\phi} $ 
by definition, and thus
\begin{multline*}
    \tau^{-1}(w_1)\circ( \tau^{-1}(w_2w_3')\circ x) - 
\tau^{-1}(w_2)\circ( \tau^{-1}(w_1w_3')\circ x)\\
=
u_1\circ ((u_2\circ u)\circ x) - u_2\circ ((u_1\circ u)\circ x) 
=
(u_1\circ (u_2\circ u))\circ x - (u_2\circ (u_1\circ u))\circ x
\end{multline*}
by \eqref{eq:SLS-2}.
The right-hand side is of length $\le 1$.
\end{enumerate}

\begin{proof}
Suppose $\mu (F) = 0$ for $F\in W^{\otimes k}$. 
Then $F$ may be presented as a linear combination of tensors 
of the form 
\[
w_1\otimes\dots \otimes w_{j-1}\otimes (ab \otimes cd - 
cb \otimes ad)\otimes w_{j+1}\otimes \dots \otimes w_k,
\]
where $w_i,a,b,c,d $ are monomials in $\Com\langle X^{(\omega )} \rangle$
and all tensor factors are of weight $-1$.
%

It remains to note that 
\begin{equation}\label{eq:ExchFactors}
(\tau^{-1}(ab), \tau^{-1}(cd), u) = (\tau^{-1}(cb), \tau^{-1}(ad), u) 
\end{equation}
holds in $\SLS\langle X\rangle $ for all appropriate 
$a,b,c,d\in \Com\langle X^{(\omega )}\rangle $ and 
for all $u\in \SLS\langle X\rangle $.
Since $\wt(ab)=\wt(cd)=-1$ and due to left-symmetry of an SLS-algebra, it is enough to consider the case when 
$\wt(a)=\wt(c)=-1$.
%

Let $x=\tau^{-1}(a)$. It follows from the definition of 
$\tau $ that for every $b$ such that $\wt(b)=0$ 
the element 
$\tau^{-1}(ab)$ considered in $\Nov \langle X \rangle $ 
is a linear combination of elements 
\[
R(x,u_1,\dots, u_m) = (\dots ((x\circ u_1)\circ u_2)\circ  \dots  )\circ u_m,\quad u_i\in N.
\]
It follows from \eqref{eq:SLS-1} that for every $u\in B$
we may write 
\[
\tau^{-1}(ab)\circ u = R(x,u_1,\dots, u_m)\circ u \in \SLS\langle X\rangle 
\]

In a similar way, one may present 
$\tau^{-1}(cd)$ in $\Nov \langle X\rangle $
as 
$R(y,v_1,\dots, v_l)$. Then 
\[
(\tau^{-1}(ab), \tau^{-1}(cd), u) 
=
(R(x,u_1,\dots, u_m), R(y,v_1,\dots, v_l), u), 
\]
\[
(\tau^{-1}(cb), \tau^{-1}(ad), u)
=
(R(y,u_1,\dots, u_m), R(x,v_1,\dots, v_l), u). 
\]
The right-hand sides of the last two relations are equal due to 
the identities of SLS-algebras. 

Indeed, let us prove the identity
\begin{equation}\label{eq:SLS-longId}
(R(x,u_1,\dots, u_m), R(y,v_1,\dots, v_l), u)
=
(R(y,u_1,\dots, u_m), R(x,v_1,\dots, v_l), u)    
\end{equation}
by induction on $m$ and $l$. 

If $m=l=0$ then \eqref{eq:SLS-longId} coincides with left symmetry. 
Assume $m=0$, $l>0$ and \eqref{eq:SLS-longId} is already proved 
for smaller $l$'s. Then \eqref{eq:SLS-2} together with left symmetry imply 
\begin{multline*}
(x,R(y,v_1,\dots, v_l), u)  
= (x,R(y,v_1,\dots, v_{l-1})\circ v_l, u)\\
= (R(y,v_1,\dots, v_{l-1}), x\circ v_l, u)
= (x\circ v_l, R(y,v_1,\dots, v_{l-1}), u)\\
= (y, R(x\circ v_l,v_1,\dots, v_{l-1}), u)
\end{multline*}
by induction. 
It remains to note that 
$R(x\circ v_l,v_1,\dots, v_{l-1})$
may be replaced here with 
$R(x,v_1,\dots, v_{l-1},v_l)$ by \eqref{eq:SLS-1}.

For $m>0$, proceed in a similar way:
\begin{multline*}
(R(x,u_1,\dots, u_m), R(y,v_1,\dots, v_l), u)
= (y, R(R(x,u_1,\dots, u_m),v_1,\dots, v_l), u)\\
 = (y, R(x,u_1,\dots, u_m,v_1,\dots, v_l), u)
 =  (y, R(x,v_1,\dots, v_l,u_1,\dots, u_m), u)\\
   = (y, R(R(x,v_1,\dots, v_l),u_1,\dots, u_m), u)
  = (R(x,v_1,\dots, v_l), R(y,u_1,\dots, u_m), u)\\
= (R(y,u_1,\dots, u_m),R(x,v_1,\dots, v_l), u).
\end{multline*}

As a result, 
\begin{multline}\label{eq:AssociatorId}
L(\tau^{-1}(w_1),\dots, \tau^{-1}(w_{j-1}), \tau^{-1}(ab),
\tau^{-1}  (cd) , \tau^{-1}( w_{j+1}),\dots \tau^{-1}( w_k),x )\\
-
L(\tau^{-1}(w_1),\dots, \tau^{-1}(w_{j-1}), \tau^{-1}(cb),
\tau^{-1}  (ad) , \tau^{-1}( w_{j+1}),\dots \tau^{-1}( w_k),x )\\
\equiv 
L(\tau^{-1}(w_1),\dots, \tau^{-1}(w_{j-1}), 
( \tau^{-1}(ab), \tau^{-1}  (cd) , u ) 
-  (\tau^{-1}(cb), \tau^{-1}  (ad), u ))
\end{multline}
modulo summands of smaller length, where 
$u=L(\tau^{-1}( w_{j+1}),\dots \tau^{-1}( w_k),x )$.
The right-hand side of 
\eqref{eq:AssociatorId} is zero by \eqref{eq:SLS-longId}.
\end{proof}

Summarizing the exposition of the section, we may state 

\begin{theorem}\label{thm:SLS-special}
For a nonempty set $X$, the subalgebra $F$ generated by $X$ in 
the free differential Perm-algebra
$\Perm\Der\langle X,d\rangle $ relative to the operation 
$f\circ g = fd(g)$, $f,g\in \Perm\Der\langle X,d\rangle $ 
is isomorphic to the free SLS-algebra generated by~$X$.
\end{theorem}

\begin{corollary}
A linear basis of the subalgebra $F$ mentioned 
in Theorem \ref{thm:SLS-special} consists 
of all monomials $u\in \Perm\<X^{(\omega )}\>$
of weight $-1$ such that $u=vx^{(n)}$, $x\in X$, $n>0$.
\end{corollary}

The proof is completely analogous to that of \cite{DzhLofwall2002} for $\Com\Der\<X,d\>$.

\section{A criterion of speciality}\label{sec:Crit}

Let us say that an SLS-algebra $(A,\circ )$ is {\em nice} 
if for every $x\in A$ there exists a linear map $\mu_x : A\circ A \to A$
such that $\mu_x (a\circ b) = (a\circ x)\circ b$ for all $a,b\in A$.

It is easy to see that $(A,\circ )$ is nice
if and only if
for every $a_i,b_i\in A$ ($i=1,\dots,n$) the equality 
$$
\sum\limits_i a_i\circ b_i=0
$$
implies 
$$
\sum\limits_i (a_i\circ x)\circ b_i = 0
$$
for all $x\in A$.

\begin{example}
A Novikov algebra is a nice SLS-algebra, 
the operator $\mu_x$ coincides with the right multiplication by~$x$.
\end{example}

\begin{example}\label{exmp:SLS-1}
Let $V$ be a linear space and let $A = T(V)/I$, 
where $I$ is the ideal spanned by 
\[
a_0a_1\dots a_n a_{n+1} - a_0a_{\sigma(1)}\dots a_{\sigma(n)} a_{n+1} ,
\quad a_i\in V, \ \sigma \in S_n.
\]
Obviously, $A$ is an associative (thus, left-symmetric) algebra 
satisfying the identities of SLS-algebras. This is not a Novikov 
algebra, but it is easy to see 
that if
$\sum\limits_i u_iv_i = 0$ in $A$ then
$\sum\limits_i u_ixv_i=0$.
So, $A$ is a nice SLS-algebra which is not Novikov.
\end{example}

\begin{example}\label{exmp:SLS-2Q}
Suppose $V$ is a 3-dimensional space with a basis $x,y,z$. 
Let $A$ be the nice SLS-algebra from Example~\ref{exmp:SLS-1}, 
and let $A' = A/(xz)$. Then $xyz\ne 0$ in $A'$, so $A'$ 
is an SLS-algebra which is not nice.
\end{example}

Example~\ref{exmp:SLS-2Q} shows us that the class of nice SLS-algebra 
is not a variety: it is not closed relative to homomorphic images.

The main purpose of this section is to prove the following criterion. 

\begin{theorem}\label{thm:NiceSpec}
A left-symmetric algebra $(A,\circ )$ 
embeds into a differential Perm-algebra in such a way that 
$a\circ b = a d(b)$ for $a,b\in A$
if and only if $A$ is a nice SLS-algebra.
\end{theorem} 

The ``only if'' part is easy: the defining identities of an SLS-algebra 
holds on a differential Perm-algebra, and the map $\mu_x$ acts as a (left) multiplication by $d(x)$: 
$$
\mu_x(a\circ b) = d(x)ad(b) = a d(x)d(b) = (a\circ x)\circ b
$$
for $a,b,x\in A$.

To prove the ``if'' part, we first construct the universal enveloping
Novikov dialgebra for a nice SLS-algebra and then apply 
Corollary~\ref{cor:Dialg-embedding}.

Suppose $(A,\circ )$ is a nice SLS-algebra.
As above, let $T(A)$ stand for the (associative) tensor algebra of 
the space $A$ (without an identity), and let $\mathcal D(A)$ 
be the left-symmetric dialgebra constructed on $T(A)$
relative to the operations defined by \eqref{eq:tensor-vdash}, \eqref{eq:tensor-dashv}.

Let $V$ stand for the right ideal of the algebra $T(A)$ generated by
\[
g(a,x) = ax - \mu_x(a), \quad a\in A\circ A\subseteq A,\ x\in A.
\]
Choose a linear basis $X$ of $A$ in the form $X = X_0\cup X_1$
where $X_0$ is a linear basis of $A\circ A$ and $X_1$ is a complement of 
$X_0$.
It is easy to see that 
\[
 V = \mathrm{Span}\, \{ax_1\dots x_n - \mu_{x_n}\dots \mu_{x_1}(a) \mid  a\in X_0, x_i\in X, i\ge 1\}.
\]
In particular, $V\cap A = 0$ since every generator of $V$ as of a linear space contains a unique principal term.

Note that 
\begin{equation}\label{eq:Mu-via-M}
 (a\circ b) f - M(afb) \in V
\end{equation}
for all $a,b\in A$, $f\in T(A)$, by the definition of~$M$.

Denote 
\[
 K = \{f\in T(A) \mid M(afb)=0\ \text{for all}\ a,b\in A \}
\]
For example, $xy-yx\in K$. 
Moreover, if $f\in K$ then for all $u,v\in T(A)\cup \{1\}$ we have 
$ufv \in K$. Indeed, 
\[
 M(aufvb) = M(M(au)fx_1\dots x_nb) = M(M(M(au)fx_1)x_2\dots x_nb) = 0 
\]
for $v=x_1\dots x_n$, $n\ge 0$. Here we used the obvious corollary of the definition of $M$:
\[
 M(uv) = M(M(u)v),\quad u,v\in T(A).
\]
In particular, if $f\in K$ and $u,v\in T(A)$ then 
$M(ufv)=0$.

Note that 
\[
h(x,y,z) = xyz + (x\circ z)\circ y -(x\circ y)z - (x\circ z)y \in K 
\]
for  all $x,y,z\in A$. Indeed, compute 
\begin{multline*}
 M(ah(x,y,z)b) 
 = M(axyz + a(x\circ z)\circ y -a(x\circ y)z - a(x\circ z)y) \circ b \\
 =(
 ((a\circ x)\circ y)\circ z + a\circ ((x\circ z)\circ y) -(a\circ (x\circ y))\circ z - (a\circ (x\circ z))\circ y
 )\circ b  \\
 =
 (
 ((a\circ x)\circ z)\circ y + a\circ ((x\circ z)\circ y) -(a\circ z)\circ (x\circ y) - (a\circ (x\circ z))\circ y
 )\circ b
 =0\circ b=0.
\end{multline*}

Let $I\subset T(A)$ stand for the linear span of all $ufv$, $u\in T(A)$, $v\in T(A)\cup\{1\}$; this is a proper 
subset of $K$. In particular, $uxy-uyx \in I$ for all $u\in T(A)$, $x,y\in A$. Hence, we may 
permute variables in $T(A)$ modulo $I$ leaving the first letter unchanged.

\begin{lemma}\label{lem:ideal1}
 The set $I$ is a two-sided ideal of $\mathcal D(A)$.
\end{lemma}

\begin{proof}
For all $a\in A$, $u,v\in T(A)$ we have 
\begin{equation}\label{eq:A-la-Poisson}
a\vdash (uv) \equiv (a\vdash u)v + u(a\vdash v) - uav\pmod I.
\end{equation}
It is easy to prove by induction on the length of $v$.

Let $f\in K$, $u\in T(A)$, $v\in T(A)\cup \{1\}$.
If $v\ne 1$ then $M(ufv)=0$ and thus 
\[
 ufv \vdash T(A) + T(A)\dashv ufv =0.
\]
Consider the case $v=1$. 
For every $a,b\in A$, left symmetry implies
\begin{multline*}
 M(aM(uf)b) =  (a\circ M(uf))\circ b \\
 =
 (a\circ (M(uf)\circ b)  + (M(uf)\circ a)\circ b - M(uf)\circ (a\circ b ) = 0
\end{multline*}
since $M(uf)\circ x = M(ufx)=0$ for all $x\in A$.
Hence, $M(uf)\in K$.

Therefore, \eqref{eq:A-la-Poisson} leads to 
\[
 uf\vdash w = M(uf)\vdash w \equiv  -(n-1)wM(uf) \equiv 0 \pmod I
\]
for every $w\in T(A)$ of length $n$. 

Similarly, 
\[
 w\dashv uf = w M(uf) \in I
\]
for every $w\in T(A)$.

This is a straightforward corollary of the definition that 
$I\vdash T(A) \subseteq I$. It remains to show $T(A)\dashv I\subseteq I$.
Choose $ufv \in I$ for $f\in K$, $w\in T(A)$, and apply \eqref{eq:A-la-Poisson} for $x=M(w)$:
\[
 w\vdash ufv \equiv (x\vdash u)fv + u(x\vdash f)v +uf(x\vdash v) -2uxfv \pmod I
\]
(for $v\ne 1$; if $v=1$ then the equation is analogous but simpler).
Obviously, we need to prove $x\vdash f \in K$ since all other summands belong to~$I$.
Indeed, consider 
\[
x\vdash af \equiv (x\circ a)f+a(x\vdash f)-axf \pmod I .
\]
Since $M(I)\circ b = 0$, we obtain
\begin{multline*}
M(a(x\vdash f)b) 
= M((x\vdash af) b) - M((x\circ a)fb) + M(axfb) \\
= M(x\vdash af) \circ b 
= (x\circ M(af))\circ b = 0
\end{multline*}
due to left symmetry. Hence, $T(A)\vdash I \subseteq I$.
\end{proof}


\begin{lemma}\label{lem:di-ideal-I}
The 
set $J=I+V$ is a two-sided ideal of the dialgebra $\mathcal D(A)$. 
\end{lemma}

\begin{proof}
Since $I$ is already known to be an ideal, it remains
to consider the right ideal $V$ of $T(A)$.
By definition, $V\dashv T(A) \subseteq V$.

Denote 
\[
g(a\circ b, x) = (a\circ b)x - (a\circ x)\circ b,\quad a,b,x\in A.
\]
Then $M(g(a\circ b, x)v)=0$ for all $v\in T(A)$ by \eqref{eq:SLS-1}, and thus 
$VA\vdash T(A) = 0$. 
For $v=1$, we may see 
$M(g(a\circ b,x))= (a\circ b)\circ x - (a\circ x )\circ b$.
Hence, 
$uM(g(a\circ b, x)) = u((a\circ b)\circ x - (a\circ x )\circ b)
\equiv u h(a,x,b) - uh(a,b,x)\equiv 0 \pmod I$.
It follows that 
\[
V\vdash T(A) \subseteq I
\]
by \eqref{eq:A-la-Poisson}.

Similarly, 
$u\dashv g(a\circ b,x)v = uM(g(a\circ b,x)v) \in I$.
 for all $u\in T(A)$, $v\in T(A)\cup \{1\}$.

To complete the proof it remains to calculate 
\begin{multline*}
c\vdash g(a\circ b, x)
=
c\vdash (a\circ b)x - c\circ ((a\circ x)\circ b) \\
=
(c\circ (a\circ b))x + (a\circ b)(c\circ x) - (a\circ b)cx
- c\circ ((a\circ x)\circ b) \\
\equiv 
(c\circ x)\circ (a\circ b) + (a\circ(c\circ x))\circ b
- ((a\circ c)\circ x)\circ b
- c\circ ((a\circ x)\circ b)
\pmod V.
\end{multline*}
The latter relation is zero in every SLS-algebra.
Then $u\vdash g(a\circ b, x)v \in J$ 
for all $u,v\in T(A)$ by \eqref{eq:A-la-Poisson}.
\end{proof}

\begin{lemma}\label{lem:Nov-quot}
For all $u,w,v \in T(A)$ we have
\[
(u\dashv v)\dashv w - (u\dashv w)\dashv v \in  J,
\]
\[
(u\vdash v)\vdash w - (u\vdash w)\dashv v \in J.
\]
In particular, $\mathcal D(A)/J$ is a Novikov dialgebra.
\end{lemma}

\begin{proof}
The first relation follows from the definition of operations:
\[
(u\dashv v)\dashv w - (u\dashv w)\dashv v 
= uM(v)M(w) - uM(w)M(v) \in I .
\]
Let us prove the second one by induction on the length of~$w$.
For  $w=a\in A$,
\[
 (u\vdash v)\vdash a = M(u\vdash v)\circ a = (M(u)\circ M(v))\circ a,
\]
\[
 (u\vdash a)\dashv v =  (M(u)\circ a)M(v) \equiv \mu_{M(v)} (M(u)\circ a) \pmod V.
\]
Hence,
 $(u\vdash v)\vdash a - (u\vdash a)\dashv v \in V\subset J$.
Next,
\begin{multline*}
 (u\vdash v)\vdash wa 
 = ((u\vdash v) \vdash w)a + w(M(u\vdash v) \circ a) -  w M(u\vdash v)a  \\
 = ((u\vdash v) \vdash w)a + w((M(u)\circ M(v)) \circ a) -  w (M(u)\circ M(v))a, 
\end{multline*}
\begin{multline*}
 (u\vdash wa)\dashv v =  (u\vdash wa)M(v) \\
 = (u\vdash w)aM(v) + w(M(u)\circ a ) M(v) - wM(u)aM(v) \\
\equiv 
(u\vdash w)M(v)a + w(M(u)\circ a ) M(v) - wM(u)aM(v) \pmod I.
\end{multline*}
Hence, 
\begin{multline*}
(u\vdash v)\vdash wa  -    (u\vdash wa)\dashv v  \\
\equiv 
((u\vdash v) \vdash w)a - (u\vdash w)M(v)a 
+ wh(M(u),a,M(v))  \\
= ((u\vdash v) \vdash w - (u\vdash w)\dashv v) a 
+ wh(M(u),a,M(v)) \equiv 0 \pmod J.
\end{multline*}
\end{proof}

\begin{remark}
Suppose $A$ is a Novikov algebra. If $T$ is an ideal of 
the left-symmetric algebra $\mathcal D(A)$
such that $\mathcal D(A)/T$ is a Novikov algebra then 
$T$ contains $uh(x,y,z)$ for all $u\in T(A)$, $x,y,z\in A$
(it follows from the proof of Lemma~\ref{lem:Nov-quot}). 
Therefore, the universal enveloping Novikov dialgebra $\mathcal N(A)$
of $A$ is actually an image of $A\oplus A^{\otimes 2}\otimes A^{\otimes 3}\subset T(A)$. In particular, if $A$ is finite-dimensional then so is $\mathcal N(A)$.
\end{remark}

\begin{lemma}\label{lem:Intrsect-zero}
 $A\cap J = \{0\}$.
\end{lemma}

\begin{proof}
Recall that a linear basis $X$ of $A$ is chosen in the form 
$X = X_0\cup X_1$
where $X_0$ is a basis of $A\circ A$ and $X_1$ is its complement.
Then $T(A)$ is the space of non-commutative polynomials in $X$. 

By definition, the space $V$ is spanned by 
\begin{equation}\label{eq:newVspan}
g(a,x_1\dots x_n) = ax_1\dots x_n  - \mu_{x_n}\dots \mu_{x_1}(a),
\quad a\in X_0,\ x_1,\dots, x_n \in X.
\end{equation}

Suppose $u\in T(A)$ is a word starting with a letter from $X_0$:
$u = (a\circ b)w$.
Then for every $f\in K$, $v\in T(A)\cup \{1\}$ we have 
\[
ufv = (a\circ b)wfv \equiv M(awfvb) = 0\pmod V
\]
by \eqref{eq:Mu-via-M}.

Assume $J=I+V$ has a nonzero intersection with $A$. 
Then we may write the following equation in $T(A)$:
\[
f_1+f_0+g= a \in A,\quad a\ne 0,
\]
where $g\in V$, $f_i$ is a linear combination of the elements like
$ufv$ with $u$ starting with a letter from $X_i$, $i=0,1$. 
As we noted above, $f_0\in V$. 
Every monomial in $f_1$ is of length at least~2 by the definition of~$I$.
All monomials that emerge in elements of $V$ start 
with a letter from $X_0$, so $f_1$ must be zero.

Therefore, we come to $a=f_0+g \in V\cap A$ which is impossible 
for $a\ne 0$.
\end{proof}

\begin{corollary}
 A nice SLS-algebra $(A,\circ )$
 may be
 embedded into a Novikov dialgebra $(\mathcal N(A),\vdash,\dashv )$ 
 in such a way  that 
 $a\vdash b = a\circ b $ for $a,b\in A$.
\end{corollary}

Indeed,
$\mathcal N(A) = \mathcal D(A)/J$ is the desired dialgebra.

\begin{remark}
Note that the generators of $I$ and $V$ represent the necessary conditions  on a Novikov dialgebra which  contains $(A,\circ)$
in such a way that 
 $a\vdash b = a\circ b $ for $a,b\in A$.
Hence, $\mathcal D(A)/J$ is the universal enveloping Novikov dialgebra of $(A,\circ)$.
\end{remark}

Now we may finish the proof of Theorem~\ref{thm:NiceSpec}.
A nice SLS-algebra $(A,\circ)$ embeds into a Novikov dialgebra 
$(\mathcal N(A),\vdash,\dashv )$, and the latter embeds into 
a differential Perm-algebra $(P,d)$ by Theorem~\ref{thm:diNov-Perm}.
Since $a\circ b = a\vdash b = ad(b) \in P$ for $a,b\in A$, 
this is a desired embedding.

\begin{corollary}
 The free SLS-algebra $\SLS\langle X\rangle $ is nice. 
\end{corollary}

\begin{corollary}
 The variety SLS is the smallest one containing the class 
 of all special left-symmetric algebras.
\end{corollary}

\end{document}